\documentclass{amsart}

\usepackage{calc}
\usepackage{tikz}
\usepackage{pgfplots}
\usetikzlibrary{shapes}
\usepackage{amsmath,amssymb}
 \usepackage[foot]{amsaddr}
 \newcommand\sn[1]{{\color{red} #1}}
  \newcommand\ma[1]{{\color{black} #1}}
 
\let\oldtocsection=\tocsection
 
\let\oldtocsubsection=\tocsubsection 
 
\let\oldtocsubsubsection=\tocsubsubsection
 
\renewcommand{\tocsection}[2]{\vspace{0.5em}\hspace{0em}\oldtocsection{#1}{#2}}
\renewcommand{\tocsubsection}[2]{\vspace{0.5em}\hspace{1em}\oldtocsubsection{#1}{#2}}
\renewcommand{\tocsubsubsection}[2]{\vspace{0.5em}\hspace{2em}\oldtocsubsubsection{#1}{#2}}
\usepackage{graphicx,cancel,xcolor,hyperref,comment,graphicx,geometry}
 
\let\originallesssim\lesssim
\DeclareRobustCommand{\lesssim}{%
  \mathrel{\mathpalette\lowersim\originallesssim}%
}
\usepackage{tikz}
\usepackage{graphicx,url,etoolbox}
\usepackage{lipsum}
\usepackage{dsfont}
\usepackage{subfigure}
\usetikzlibrary{hobby,patterns}

\usepackage{fancyhdr}
\usepackage{lastpage}

\newtheorem{theoreme}{Theorem}[section]
\newtheorem{pro}[theoreme]{Proposition}
\newtheorem{lemma}[theoreme]{Lemma}

\newtheorem{definition}[theoreme]{Definition}
\theoremstyle{definition}

\numberwithin{equation}{section}

 \renewenvironment{proof}{{\bfseries \noindent Proof.}}{\demo}
\newcommand\xqed[1]{%
  \leavevmode\unskip\penalty9999 \hbox{}\nobreak\hfill
  \quad\hbox{#1}}
\newcommand\demo{\xqed{$\square$}}

\hypersetup{bookmarks, colorlinks, urlcolor=blue, citecolor=red, linkcolor=blue, hyperfigures, pagebackref,
    pdfcreator=LaTeX, breaklinks=true, pdfpagelayout=SinglePage, bookmarksopen=true,bookmarksopenlevel=2}

\usepackage{comment}
\def\u2{\u^2}
\def\u3{\u^3}
\def\u4{\u^4}
\def\u5{\u^5}
\def\y1{\y^1}
\def\y2{\y^2}
\def\y3{\y^3}
\def\y4{\y^4}
\def\y5{\y^5}

\def\R{\mathbb R}

\def\HH{\mathcal H}

\def\AA{\mathcal A}

\def\la {{\lambda}}

\newcommand {\nc}   {\newcommand}
\nc {\be}   {\begin{equation}} \nc {\ee}   {\end{equation}} \nc
{\beq}  {\begin{eqnarray}} \nc {\eeq}  {\end{eqnarray}} \nc {\beqs}
{\begin{eqnarray*}} \nc {\eeqs} {\end{eqnarray*}}
\def\edc{\end{document}}

\providecommand{\abs}[1]{\lvert#1\rvert}
 
\DeclareMathOperator{\divv}{div}  

\usepackage{tikz}
\usetikzlibrary{decorations.pathmorphing,patterns,scopes,intersections,calc}
\usepackage{caption}
         
\begin{document}
\title[\fontsize{7}{9}\selectfont  ]{Stabilization  of  coupled wave equations with viscous damping on cylindrical and  non-regular domains:  Cases without the geometric control condition}
\author{Mohammad Akil$^{1}$, Haidar Badawi$^{1}$, Serge Nicaise$^{1}$ and Virginie R\'egnier$^{1}$ }
\address{$^1$Universit\'e Polytechnique  Hauts-de-France, C\'ERAMATHS/DEMAV, Le Mont Houy 59313 Valenciennes Cedex 9-France}
\email{mohammad.akil@uphf.fr, haidar.badawi@etu.uphf.fr, serge.nicaise@uphf.fr, virginie.regnier@uphf.fr}

\keywords{coupled wave equations; viscous damping; $C_0$-semigroup; polynomial stability; cylindrical domains.}

\setcounter{equation}{0}
\begin{abstract}
In this paper, we investigate the direct and indirect stability of  locally coupled wave equations with local viscous damping on cylindrical and non-regular domains without any geometric control condition.  If only one equation is  damped, we prove that the energy of our system decays polynomially with the rate $t^{-\frac{1}{2}}$ if the two waves have the same speed of propagation, and with rate $t^{-\frac{1}{3}}$ if the two waves do not propagate at the same speed. Otherwise, in case of two damped equations, we prove a polynomial energy decay rate of order $t^{-1}$.
\end{abstract}
\maketitle
\pagenumbering{roman}
\maketitle
\tableofcontents
\pagenumbering{arabic}
\setcounter{page}{1}
\section{Introduction} 
\noindent 
Let $\omega$ be a non empty open set of $\R^{N-1}$ with a Lipschitz boundary, with $N\geq 2$,
and consider the cylindrical domain $\Omega=(0,1)\times \omega$. 
In this domain,  we consider the following    strongly  coupled  wave equations 
\begin{equation}\label{SYST1}
\left\{\begin{array}{lll}
\phi_{tt}(X,t)-\Delta \phi (X,t)+b(x)\phi_t (X,t)+c(x)\psi_t (X,t)&=&0,\ (X,t) \in  \Omega\times (0,\infty),\\[0.1in]
\psi_{tt}(X,t)-a\Delta \psi (X,t)+d(x)\psi_t (X,t)-c(x)\phi_t(X,t)&=&0,\ (X,t) \in  \Omega\times (0,\infty),\\[0.1in]
\phi(X,t)=\psi (X,t) &=&0,\ (X,t) \in   \partial \Omega \times (0,\infty),\\[0.1in]
(\phi(X,0),\phi_t(X,0))&=&(\phi_0(X),\phi_1(X)),\ X \in  \Omega,\\[0.1in] 
(\psi(X,0),\psi_t(X,0))&=&(\psi_0(X),\psi_1(X)),\ X \in  \Omega.
\end{array}
\right.
\end{equation}
where $a$ is a positive real number,  $b,c,d\in L^{\infty}\left(0,1\right),$ {\color{black} such that 
\begin{equation}\tag{${\rm b,c}$}\label{b,c}
	\left\{\begin{array}{l}
		b(x)\geq b_0>0\ \ \text{in}\ \ (\alpha_1,\alpha_4)
		\ \
		\text{and}\ \
		b(x) \geq 0\ \ \text{in}\ \ (0,1)\backslash (\alpha_1,\alpha_4),
		\\[0.1in]
		|c(x)|\geq c_0> 0\ \ \text{in}\ \ (\alpha_2,\alpha_3)  \ \
		\text{and}\ \ c(x)= 0\ \text{in}\ (0,1)\backslash (\alpha_2,\alpha_3), 
	\end{array}
	\right.
\end{equation}
where $0\leq \alpha_1<\alpha_2<\alpha_3<\alpha_4\leq 1$,}
 and  as usual $\partial \Omega$ is the boundary of $\Omega$. In the whole paper, $X=(x,x_1,\cdots,x_{N-1})$ is the generic variable in $\Omega $, where the first variable $x$ runs in $(0,1)$. In this paper, we study the indirect or direct stability of system \eqref{SYST1} by assuming  that
{\color{black}
\begin{equation}\tag{${\rm LCD1}$}\label{CASE1}
\eqref{b,c} \ \text{holds} \ \ \text{and} \ \ d=0 \ \ \text{in} \ \ (0,1),
\end{equation}
or 
\begin{equation}\tag{${\rm LCD2}$}\label{CASE2}
\eqref{b,c} \ \text{holds}, \ \ \text{and} \ \   d(x)\geq d_0>0\ \ \text{in}\ \ (\alpha_1,\alpha_4)
\ \
\text{and}\ \
d(x) \geq 0\ \ \text{in}\ \ (0,1)\backslash (\alpha_1,\alpha_4).
\end{equation}
 The case of global interior dampings (i.e. the case when  $b$ and/or $d$  uniformly positive in $(0,1)$) is of course allowed, but our main interest concerns local dampings, corresponding to the case when $b$ and $d$ are zero on different open subsets of $(0,1)$.
}

\noindent In 2005, Liu and Rao in \cite{RaoLiu01} have studied the stability of the wave equation on a square $(0,\pi)^2$ with local viscous damping, by considering the following damping  region  
$$
\{(x,y)\in (0,\pi)^2  \ | \ a<x<b \ \ \text{and} \ \ 0<y<\pi \},
$$
where $a$ and $b$ are  real numbers such that $0\leq a<b\leq \pi$ and $b-a<\pi$. They established a polynomial energy decay rate of order $t^{-1}$. In 2017, Stahn in \cite{Stahn2017} has studied the stability of the wave equation on a square $(0,1)^2$ with local viscous damping, by considering the following damping  region 
$$
\{(x,y)\in (0,1)^2  \ | \ 0<x<\sigma \ \ \text{and} \ \ 0<y<1 \},
$$
where $\sigma$ is a positive real number. He established an optimal polynomial energy decay rate of order $t^{-\frac{4}{3}}$. In 2019, Batty {\it et al.} in \cite{BPS} have  studied the stability of the wave-heat system on a rectangular domain. They established an optimal polynomial energy decay rate of order $t^{-\frac{2}{3}}$. In 2021, Yu and Han in \cite{doi:10.1137/20M1332499} have  studied the stability of the wave equation on a cuboidal domain via Kelvin-Voigt damping. They established an optimal energy decay rate of order $t^{-\frac{1}{3}}$. In these last three papers, the main ingredient is the separation of variables.
 In 2020, Hayek {\it et al.} in \cite{Hayek} have  studied the stability of  weakly coupled wave equations with Kelvin-Voigt damping on a square. They established, by taking the result in \cite{RaoLiu01} (resp. \cite{Stahn2017}) as an auxiliary problem, a polynomial energy decay rate of order $t^{-\frac{1}{6}}$ (resp. $t^{-\frac{1}{5}}$). However,  in 2021,  Akil {\it et al.} in \cite{akil2021ndimensional} have  studied the stability of  strongly coupled wave equations with Kelvin-Voigt damping on a square. They established by taking the result in \cite{RaoLiu01} (resp. \cite{Stahn2017}) as an auxiliary problem, a polynomial energy decay rate of order $t^{-\frac{1}{5}}$ (resp. $t^{-\frac{1}{4}}$). In  2019, Kassem {\it et al. } in \cite{kassem} have  studied the local indirect stabilization of multidimensional coupled wave equations under geometric conditions by considering system \eqref{SYST1} where $\Omega$ is a nonempty open set of $\R^N$ with a boundary $\Gamma$ of class $C^2$  and 
\begin{equation}
	\left\{\begin{array}{l}
	d(x)=0 \ \ \text{in} \ \ \Omega,\ \ b,c\in W^{1,\infty}\left(\Omega\right);\ \    \text{and}\\[0.1in]
	b(x)\geq 0\ \ \text{in}\ \ \Omega, \quad
		b(x)>0\ \ \text{in}\ \ \omega_b \subset \Omega, \ \ \text{and}\ \ c(x)\neq 0\ \ \text{in}\ \ \omega_c \subset \Omega, \ \ \text{where}\\[0.1in]
	\omega_b \cap \omega_c :=\omega_{b,c} \neq \emptyset,
	\end{array}
	\right.
\end{equation}
such that $\omega_{b,c}$ satisfies the piecewise multiplier geometric condition  (PMGC in short) introduced by K. Liu in \cite{liu-PMGC}. They established an exponential energy decay rate if the two waves have the same speed of propagation (i.e. $a=1$). In case of different speed propagation (i.e. $a\neq 1$), they obtained an optimal polynomial energy decay rate of order $t^{-1}$. While  in \cite{kassem2}, the authors have  studied the local indirect stabilization of multidimensional coupled wave equations under geometric conditions, by considering system \eqref{SYST1} where $\Omega$ is a nonempty open set of $\R^N$ with a boundary $\Gamma$ of class $C^2$  and 
\begin{equation}
	\left\{\begin{array}{l}
		d(x)=0 \ \ \text{in} \ \ \Omega,\ \ b,c\in W^{1,\infty}\left(\Omega\right); \text{and}\\[0.1in]
 b(x)\geq 0\ \ \text{in}\ \ \Omega,\ \		b(x)>0\ \ \text{in}\ \ \omega_b \subset \Omega \ \ \text{and}\ \ c(x)\neq 0\ \ \text{in}\ \ \omega_c \subset \omega_b,
	\end{array}
	\right.
\end{equation}
such that $\omega_c$ satisfies the geometric control condition  (GCC in short) introduced by Rauch and Taylor in \cite{RTPGCC} for manifolds without boundaries and by Bardos, Lebeau and Rauch in \cite{BLJ-GCC} for domains with boundaries. They established an exponential energy decay rate if the two waves have the same speed of propagation (i.e. $a=1$).\\ \linebreak
According to the references cited above, it should be noted that,  in many cases, the GCC is an important hypothesis for  coupled wave systems to achieve exponential stabilization.  Hence, this work focuses on the following question: to what extent the strongly coupled wave equations on a general  non-regular domain can be stabilized  under viscous damping if the support of the damping does not satisfy the GCC. We answer to this question in the case of  cylindrical domains described above by combining an orthogonal basis approach (separation of variables) and a new frequency multiplier method.\\ \linebreak 
This paper is organized as follows: In Subsection \ref{WPP}, we prove the well-posedness of our system by using semigroup approach. In Subsection  \ref{SS-NR}, following a general criteria of Arendt and Batty, we show the strong stability of system \eqref{SYST1}. In Section  \ref{PS-NR}, by combining an orthonormal basis decomposition  with frequency-multiplier techniques, we prove a polynomial energy decay rate of order:
$$
\left\{
\begin{array}{llr}
\displaystyle{t^{-\frac{1}{2}}}&\text{if} & a=1\ \text{and}\ \eqref{CASE1}\ \text{holds},\\[0.1in]
\displaystyle{t^{-\frac{1}{3}}}&\text{if} & a\neq 1\ \text{and}\ \eqref{CASE1}\ \text{holds},\\[0.1in]
\displaystyle{t^{-1}}&\text{if}&\ \eqref{CASE2}\ \text{holds}.\\[0.1in]
\end{array}
\right.
$$

\section{Well-Posedness and Strong Stability}\label{WPSS}
\subsection{Well-Posedness}\label{WPP}\noindent In this section, we will establish the well-posedness of system \eqref{SYST1} by using semigroup approach. Here and below, we set 

$$
\omega_b:=(\text{supp}\, b)^\circ \times \omega, \  \omega_c:=(\alpha_2,\alpha_3)\times \omega,\ \text{and}\ 
	\omega_d:=\left\{\begin{array}{lll}
\displaystyle \emptyset & \text{if} & \eqref{CASE1} \ \text{holds}, \vspace{0.15cm}\\
\displaystyle (\text{supp}\, d)^{\circ}\times \omega  & \text{if}& \eqref{CASE2} \ \text{holds}.
	\end{array}\right.
$$
The energy of system  \eqref{SYST1} is given by 
\begin{equation}\label{ENERGY}
E(t)=\frac{1}{2}\int_{\Omega}\left(a\abs{\nabla \phi}^2+\abs{\phi_t}^2+\abs{\nabla \psi}^2+\abs{\psi_t}^2\right)dX.
\end{equation}
A straightforward computation gives 
\begin{equation}\label{DENERGY}
\frac{d}{dt}E(t)=-\int_{\Omega}b\abs{\phi_t}^2dX-\int_{\Omega} d \abs{\psi_t}^2dX\leq 0,
\end{equation}
which indicates that the energy of system \eqref{SYST1} is dissipative. Now, let us define the energy space $\mathcal{H}$ by 
$$
\mathcal{H}=\left(H_0^1(\Omega)\times L^2(\Omega)\right)^2,
$$
that is a Hilbert space, equipped with the inner product defined by 
$$
\left<\Phi,\Phi_1\right>_{\mathcal{H}}=\int_{\Omega}\left(a\nabla\phi\cdot \nabla \overline{\phi}_1+v\overline{v_1}+\nabla\psi\cdot \nabla \overline{\psi}_1+z\overline{z_1}\right)dX,
$$
for all $\Phi=(\phi,v,\psi,z)^\top$ and $\Phi_1=(\phi_1,v_1,\psi_1,z_1)^\top$ in $\mathcal{H}$. The expression $\|\cdot\|_{\mathcal{H}}$ will denote the corresponding norm. We define the unbounded linear operator $\mathcal{A}:D(\mathcal{A})\subset \mathcal{H}\to \mathcal{H}$ by 
\begin{equation}\label{Domain}
D(\mathcal{A})= [D(\Delta_{\rm Dir})\times   H_0^1(\Omega)]^2,
\end{equation}
and 
$$
\mathcal{A}\Phi=\left(v,\Delta \phi-bv-cz,z,a\Delta \psi-dz+cv\right)^\top, \  \forall  \Phi=(\phi,v,\psi,z)^\top\in D(\mathcal{A}),
$$
where $D(\Delta_{\rm Dir}):=\{u\in H_0^1(\Omega)\, :\, \Delta u\in L^2(\Omega)\}.$
\\ \linebreak
If $\Phi=(\phi,\phi_t,\psi,\psi_t)^\top$ is the state of system \eqref{SYST1}, then this system is transformed into the first order evolution equation on the Hilbert space $\mathcal{H}$ given by 
\begin{equation}\label{evolution}
\Phi_t=\mathcal{A}\Phi,\quad \Phi(0)=\Phi_0,
\end{equation}
where $\Phi_0=(\phi_0,\phi_1,\psi_0,\psi_1)^\top$.
\begin{pro}
	{\rm
The unbounded linear operator $\mathcal{A}$ is m-dissipative in the energy space $\mathcal{H}$. }
\end{pro}
\begin{proof}
For all $\Phi=(\phi,v,\psi,z)^\top\in D(\AA)$, we have 
\begin{equation}\label{diss}
\Re\left<\AA \Phi,\Phi\right>_{\mathcal{H}}=-\int_{\Omega}b\abs{v}^2dX-\int_{\Omega} d \abs{z}^2dX\leq0 ,
\end{equation}
which implies that $\AA$ is dissipative. Now, let us prove that $\AA$ is maximal. For this aim, let $F=(f_1,f_2,f_3,f_4)^\top\in \mathcal{H}$, we look for a unique solution   $\Phi=(\phi,v,\psi,z)^\top\in D(\AA)$   of
\begin{equation}\label{m-diss1}
-\AA \Phi=F.
\end{equation}
Equivalently, we have the following system 
\begin{equation}\label{m-diss2}
-v=f_1\quad \text{and}\quad -z=f_3
\end{equation}
and 
\begin{equation}\label{m-diss3}
-\Delta \phi+bv+cz=f_2\quad \text{and}\quad -a\Delta \psi+dz-cv=f_4.
\end{equation}
Substituting \eqref{m-diss2} in \eqref{m-diss3}, we get 
\begin{eqnarray}
-\Delta \phi=f_2+bf_1+cf_3,\label{m-diss4}\\
-a\Delta \psi=f_4+df_3-cf_1,\label{m-diss5}
\end{eqnarray}
with full Dirichlet boundary conditions 
\begin{equation}\label{1m-diss3}
\phi=\psi=0\quad \text{on}\quad \partial\Omega. 
\end{equation}
Let $(\zeta,\xi)\in H_0^1(\Omega)\times H_0^1(\Omega)$. Multiplying \eqref{m-diss4} and \eqref{m-diss5} by $\overline{\zeta}$ and $\overline{\xi}$, integrating over $\Omega$, using formal integration by parts and the definition of $b$ and $c$, and adding the two equations  we get
\begin{equation}\label{m-diss6}
\beta ((\phi,\psi),(\zeta,\xi))=L(\zeta,\xi),\quad \forall (\zeta,\xi)\in H_0^1(\Omega)\times H_0^1(\Omega),
\end{equation}
where 
\begin{equation*}
\beta ((\phi,\psi),(\zeta,\xi))=\int_{\Omega}\nabla \phi\cdot \nabla\overline{\zeta}dX+a\int_{\Omega}\nabla \psi\cdot \nabla\overline{\xi}dX.
\end{equation*}
and 
\begin{equation*}
L(\zeta,\xi)=\int_{\Omega}\left(f_2+bf_1+cf_3\right)\overline{\zeta}dX+\int_{\Omega}\left(f_4+df_3-cf_1\right)\overline{\xi}dX.
\end{equation*}
It is easy to see that, $\beta$ is a sesquilinear, continuous and coercive form on  $H_0^1(\Omega)\times H_0^1(\Omega)$ and $L$ is an antilinear and continuous form on $H_0^1(\Omega)\times H_0^1(\Omega)$. Then, it follows by Lax-Millgram theorem that \eqref{m-diss6} admits a unique solution $(\phi,\psi)\in H_0^1(\Omega)\times H_0^1(\Omega)$. By taking test functions  $(\zeta,\xi)\in \left(D(\Omega)\right)^2$, we see that \eqref{m-diss4}-\eqref{1m-diss3} hold in the distributional sense, from which we deduce that $\phi,\psi\in D(\Delta_{\rm Dir})$. Consequently, $\Phi=(\phi,-f_1,\psi,-f_3)^\top\in D(\AA)$ is the unique solution of \eqref{m-diss1}. Then, $\AA$ is an isomorphism, and since $\rho(\AA)$ is an open set of $\mathbb{C}$ (see Theorem 6.7 [Chapter III] in \cite{Kato01}), we easily get $R(\la I-\AA)=\HH$ for a sufficiently small $\la>0$. This, together with the dissipativeness of $\AA$, implies that $D(\AA)$ is dense in $\HH$ and that $\AA$ is $m-$dissipative in $\HH$ (see Theorems 4.5 and 4.6 in \cite{Pazy01}). The proof is thus complete.
\end{proof}

\noindent Thus, according to Lumer-Phillips Theorem (see \cite{Pazy01}), the operator  $\AA$ generates a $C_0-$semigroup of contractions $\left(e^{t\AA}\right)_{t\geq 0}$. Then,  the solution of the Cauchy problem \eqref{evolution} admits the following representation 
$$
\Phi(t)=e^{t\mathcal{A}}\Phi_0,\quad t\geq 0,
$$
which leads to the well-posedness of \eqref{evolution}. Hence, we have the following result. 
\begin{theoreme}
	{\rm
Let $\Phi_0\in \mathcal{H}$, then system \eqref{evolution} admits a unique weak solution $\Phi$ that satisfies
$$
\Phi\in C\left(\mathbb{R}_+,\mathcal{H}\right).
$$
Moreover, if $\Phi_0\in D(\mathcal{A})$, then problem \eqref{evolution} admits a unique strong solution $\Phi$ that  satisfies
$$
\Phi\in C^1\left(\R_{+},\mathcal{H}\right)\cap C\left(\R_+,D(\mathcal{A})\right).
$$}
\end{theoreme}
\subsection{Strong Stability}\label{SS-NR}
In this subsection, we will prove the strong stability of system \eqref{SYST1}. The main result of this subsection is the following theorem. 
\begin{theoreme}\label{Strong}
	{\rm
Assume that \eqref{CASE1} or \eqref{CASE2} holds. Then, the $C_0-$semigroup of contraction $\left(e^{t\mathcal{A}}\right)_{t\geq 0}$ is strongly stable in $\mathcal{H}$, i.e., for all $\Phi_0\in \mathcal{H}$, the solution of \eqref{evolution} satisfies 
$$
\lim_{t\to \infty}\|e^{t\mathcal{A}}\Phi_0\|_{\mathcal{H}}=0.
$$}
\end{theoreme}
\begin{proof}
Since the resolvent of $\mathcal{A}$ is compact in $\mathcal{H}$, then according to Arendt-Batty theorem see \cite[p. 837]{Arendt01}, system \eqref{SYST1} is strongly stable if and only if $\mathcal{A}$ does not have pure imaginary eigenvalues that is $\sigma(\mathcal{A})\cap i\mathbb{R}=\emptyset$. Since from subsection \ref{WPP}, we already have $0\in \rho(\mathcal{A})$, it remains to show that $\sigma(\mathcal{A})\cap i\mathbb{R}^{\ast}=\emptyset$. For this aim, suppose by contradiction that there exists a real number $\lambda\neq 0$ and $\Phi=(\phi,v,\psi,z)^{\top}\in D(\mathcal{A})\backslash \{0\}$ such that 
\begin{equation}\label{AUiLU}
\mathcal{A}\Phi=i\la \Phi.
\end{equation}
Detailing \eqref{AUiLU}, we get the following system 
\begin{eqnarray}
v&=&i\lambda \phi\hspace{1cm}\text{in}\quad  \Omega,\label{kvss4}\\ \noalign{\medskip}
\Delta \phi-bv-c z&=&i\lambda v\hspace{1cm}\text{in }\quad \Omega,\label{kvss5}\\ \noalign{\medskip}
z&=&i\lambda \psi\hspace{1cm} \text{in}\quad\Omega,\label{kvss6}\\ \noalign{\medskip}
a\Delta \psi-dz+c  v&=&i\la z\hspace{1cm}\text{in}\quad \Omega\label{kvss7}.
\end{eqnarray}
From \eqref{diss} and \eqref{AUiLU}, we have 
\begin{equation}\label{str-1}
	0=\Re\left(i\la \|\Phi\|^2_{\mathcal{H}}\right)=\Re\left(\left<\mathcal{A}\Phi,\Phi\right>_{\mathcal{H}}\right)=-\int_{\Omega}b\abs{v}^2dX-\int_{\Omega}d\abs{z}^2dX\leq 0.
\end{equation}
Thus, from \eqref{kvss4}, \eqref{kvss6}, \eqref{str-1} and the fact that $\la \neq 0$, we have
\begin{equation}\label{str-2}
\left\{	\begin{array}{lll}
		\displaystyle \phi= v=0 \  \ \text{in} \ \ \omega_b, \quad  &\text{if}& \ \eqref{CASE1} \ \text{holds},\\[0.1in]
		\displaystyle \phi=v=0 \ \ \text{in} \ \ \omega_b \ \ \text{and} \ \ \psi=z=0 \ \ \text{in} \ \ \omega_d, \quad &\text{if}& \ \eqref{CASE2}\  \text{holds}.
	\end{array}\right.
\end{equation}
If  \eqref{CASE1} holds. Then, from \eqref{kvss4}-\eqref{kvss5}, \eqref{str-2} and the fact that $\la \neq 0$, we get
\begin{equation}\label{2.20}
cz=0\ \ \text{in}\ \ \omega_b \supset \omega_c \ \ \text{and consequently} \ \ \psi=z=0 \ \ \text{in} \ \ \omega_c.
\end{equation}
Inserting \eqref{kvss4} and \eqref{kvss6} in \eqref{kvss5} and \eqref{kvss7}, respectively, then using \eqref{str-2} and \eqref{2.20}, we obtain
\begin{equation}\left\{
\begin{array}{lll}
\la ^2\phi+\Delta \phi=0 & \text{in}& \Omega,\\[0.1in]
\la ^2\psi+a\Delta \psi=0 & \text{in}& \Omega,\\[0.1in]
\phi=0 &\text{in} &\omega_b\subset \Omega,\\[0.1in]
\psi=0 &\text{in} &\omega_c\subset \Omega.
\end{array}\right.
\end{equation}
Using the unique continuation theorem, we get
\begin{equation}\label{2.22}
\phi=\psi=0\quad \text{in}\quad \Omega.
\end{equation}
Finally, from \eqref{kvss4}, \eqref{kvss6} and \eqref{2.22}, we deduce that 
$$\Phi=0.$$

\noindent Let us continue the proof in case that \eqref{CASE2} holds.\,Then from \eqref{str-2},  \eqref{kvss4}-\eqref{kvss7} and the fact that $\omega_c \subset \omega_b\cap\omega_d\neq \emptyset$, again using the unique continuation theorem it is easy to conclude that
$$
\Phi=0 \ \ \text{in} \ \ \Omega.
$$

\noindent The proof has been completed.
\end{proof}

\section{Polynomial Stability}\label{PS-NR}
\noindent The aim of this section is to prove the polynomial stability of the system \eqref{SYST1}. Our main results in this section are the following theorems. 
\begin{theoreme}\label{G-P}
	{\rm
Assume that \eqref{CASE1} holds. Then, there exists a constant $C>0$ independent of $\Phi_0$, such that the energy of system \eqref{SYST1} satisfies  the following estimation
\begin{equation}\label{Polest}
E(t)\leq \frac{C}{t^{\frac{2}{\ell}}}\|\Phi_0\|^2_{D(\mathcal{A})},\quad \forall t> 0,\quad \forall \Phi_0\in D(\AA),
\end{equation}
where 
\begin{equation}\label{ell}
\ell=\left\{\begin{array}{lll}
4&\text{if}&a=1,\\[0.1in]
6&\text{if}&a\neq 1. 
\end{array}
\right.
\end{equation}}
\end{theoreme}
\begin{theoreme}\label{1G-P}
	{\rm
Assume that \eqref{CASE2} holds. Then, there exists a constant $C>0$ independent of $\Phi_0$, such that the energy of system \eqref{SYST1} satisfies  the following estimation
\begin{equation}\label{Polest1}
E(t)\leq \frac{C}{t}\|\Phi_0\|^2_{D(\mathcal{A})},\quad \forall t> 0,\quad \forall \Phi_0\in D(\AA).
\end{equation}}
\end{theoreme}
\noindent To prove them, let us first recall the following   necessary and sufficient condition on the polynomial stability of semigroup proposed by  Borichev-Tomilov in \cite{Borichev01} (see also \cite{Batty01}, \cite{RaoLiu01},
and the recent paper   \cite{ROZENDAAL2019359}).
\begin{theoreme}\label{GENN}
	{\rm
Assume that $A$ is the generator of a strongly continuous semigroup of contractions $\left(e^{t A}\right)_{t\geq 0}$ on a Hilbert space $H$. If 
\begin{equation}\label{irrhoA}
i\mathbb{R}\subset \rho(A),
\end{equation} then for a fixed $\ell>0$ the following conditions are equivalent 
\begin{equation}\label{limsup}
\limsup_{\la \in \R, \  |\la|\to \infty} |\la|^{-\ell}\|(i\la I-A)^{-1}\|_{\mathcal{L}(\mathcal{H})}<\infty,
\end{equation}
\begin{equation}
\|e^{tA}X_0\|_H^2\leq \frac{C}{t^{\frac{2}{\ell}}}\|X_0\|^2_{D(A)},\ \ X_0\in D(A),\ \text{for some}\ C>0.
\end{equation}}
\end{theoreme}
\noindent According to Theorem \ref{GENN}, to prove Theorems \ref{G-P}, \ref{1G-P}, we need to prove that  \eqref{irrhoA} and \eqref{limsup} hold,
where $\ell$ is defined in \eqref{ell} if \eqref{CASE1} holds, and $\ell=2$ if \eqref{CASE2} holds.   As condition \eqref{irrhoA} is already proved in Theorem \ref{Strong},  we only need to prove condition \eqref{limsup}. Here we use a contradiction argument. Namely, suppose that \eqref{limsup} is false, then there exists $\{(\la_n,\Phi^{(n)}:=(\phi^{(n)},v^{(n)},\psi^{(n)},z^{(n)}))\}_{n\geq 1}\subset \mathbb{R}^{\ast}_{+}\times D(\mathcal{A})$ with 
\begin{equation}\label{Polesteq0}
\la_n\to \infty\ \ \text{as}\ \ n\to \infty\ \ \text{and}\ \ \|\Phi^{(n)}\|_{\mathcal{H}}=\|(\phi^{(n)},v^{(n)}, \psi^{(n)},z^{(n)})\|_{\mathcal{H}}=1,\ \forall n\in \mathbb{N},
\end{equation}
such that 
\begin{equation}\label{Polesteq}
\la_n^{\ell}\left(i\la_n-\AA\right)\Phi^{(n)}=F^{(n)}:=(f^{1,(n)},f^{2,(n)},f^{3,(n)},f^{4,(n)})\to 0\quad \text{in}\quad \mathcal{H}, \ \text{as} \ n \to \infty. 
\end{equation}
For simplicity, we drop the index $n$. Detailing \eqref{Polesteq}, we get 
\begin{equation}\label{Polesteg1}
\left\{\begin{array}{lll}
i\la\phi-v =\la^{-\ell}f^1\to 0&\text{in} & H_0^1(\Omega),\\[0.1in]
i\la v-\Delta \phi +bv +cz =\la^{-\ell}f^2\to 0&\text{in} & L^2(\Omega),\\[0.1in]
i\la\psi-z=\la^{-\ell}f^3\to 0&\text{in} & H_0^1(\Omega),\\[0.1in]
i\la z-a\Delta \psi+dz-cv=\la^{-\ell}f^4 \to 0&\text{in}& L^2(\Omega).
\end{array}
\right.
\end{equation}
Here we will check the condition \eqref{limsup} by finding a contradiction with \eqref{Polesteq0} by showing $\|\Phi\|_{\mathcal{H}}=o(1)$. The technique of the proof is related to the orthonormal basis decomposition combined with a new frequency multiplier technique. To this aim, let  $\{e_j\}_{j\in \mathbb{N}^*}$ be the orthonormal basis  of the Laplace operator with Dirichlet boundary conditions in $\omega$ such that
$$
-\Delta e_j=\mu_j^2e_j,
$$
and $\mu_j\to \infty$ when $j\to \infty$. We may expand $\phi$ into a series of the form 
\begin{equation}\label{Dec-phi}
\phi(X)=\sum_{j=1}^{\infty} \phi_j(x)e_j(x_1,\cdots,x_{N-1}),\quad X=(x,x_1,\cdots,x_{N-1})\in \Omega.
\end{equation}
Similarly, $v,\psi,z,f^1,f^2,f^3$ and $f^4$ can be decomposed into a form of series expansion similar to that in \eqref{Dec-phi} with, respectively, the coefficients   $v_j(x),\psi_j(x),z_j(x),f^1_j(x),f^2_j(x),f^3_j(x),f^4_j(x)$. This gives rise to functions 
\begin{equation}\label{newdomain}
(\phi_j,v_j,\psi_j,z_j)\in \left(\left(H^2(0,1)\cap H_0^1(0,1)\right)\times H_0^1(0,1)\right)^2\ \ \text{and}\ \ (f^1_j,f^2_j,f^3_j,f^4_j)\in \left(H_0^1(0,1)\times L^2(0,1)\right)^2. 
\end{equation}
Using the orthonormality of the set $\{e_j\}_{ j\in \mathbb{N^{\ast}}}$, system \eqref{Polesteg1} turns into the system of one-dimensional equations 
\begin{eqnarray}
i\la \phi_j-v_j&=&\la^{-\ell}f^1_j,\label{eq1-pol}\\[0.1in]
i\la v_j-\phi_j^{''}+\mu_j^2\phi_j+bv_j+cz_j&=&\la^{-\ell}f^2_j,\label{eq2-pol}\\[0.1in]
i\la \psi_j-z_j&=&\la^{-\ell}f^3_j,\label{eq3-pol}\\[0.1in]
i\la z_j-a\psi_j^{''}+a\mu_j^2\psi_j+dz_j-cv_j&=&\la^{-\ell}f^4_j.\label{eq4-pol}
\end{eqnarray}
where  " $'$ " represents the derivative with respect to $x$. System \eqref{eq1-pol}-\eqref{eq4-pol} is subjected to the following boundary conditions
$$
\phi_j(0)=\phi_j(1)=\psi_j(0)=\psi_j(1)=0.
$$
\noindent Note that from the orthonormal basis decomposition, we have 
\begin{equation*}
\left\{\begin{array}{l}
\displaystyle
\|F\|_{\mathcal{H}}^2=\sum_{j=1}^{\infty}\left(\|(f^1_j)'\|^2+\mu_j^2\|f^1_j\|^2+\|f^2_j\|^2+a\|(f^3_j)'\|^2+a\mu_j^2\|f^3_j\|^2+\|f^4_j\|^2\right),\\[0.1in]
\displaystyle
\|\Phi\|^2_{\HH}=\sum_{j=1}^{\infty}\left(\|(\phi_j)'\|^2+\mu_j^2\|\phi_j\|^2+\|v_j\|^2+a\|(\psi_j)'\|^2+a\mu_j^2\|\psi_j\|^2+\|z_j\|^2\right),
\end{array}
\right.
\end{equation*}
where $\|\cdot\|:=\|\cdot\|_{L^2(0,1)}$ and $\|\cdot\|_\infty:= \|\cdot\|_{L^\infty (0,1)}$. Inserting \eqref{eq1-pol} and \eqref{eq3-pol} respectively in \eqref{eq2-pol} and \eqref{eq4-pol}, we get 
\begin{eqnarray}
\la^2\phi_j+\phi^{''}_j-\mu_j^2\phi_j-i\la b\phi_j-i\la c\psi_j=F^1_j,\label{COMB1}\\[0.1in]
\la^2\psi_j+a\psi_j^{''}-a\mu_j^2\psi_j-i\la d\psi_j+i\la c\phi_j=F^2_j,\label{COMB2}
\end{eqnarray}
where
\begin{equation}\label{F1jF2j}
F^1_j=-\left(\frac{f^2_j}{\la^{\ell}}+\frac{if^1_j}{\la^{\ell-1}}+\frac{bf^1_j}{\la^{\ell}}+\frac{cf^3_j}{\la^{\ell}}\right)\quad \text{and}\quad F^2_j=-\left(\frac{f^4_j}{\la^{\ell}}+i\frac{f^3_j}{\la^{\ell-1}}\frac{df^3_j}{\la^{\ell}}-\frac{cf^1_j}{\la^{\ell}}\right). 
\end{equation}
Before going on, let us first give the consequence of the dissipativeness property on the   solution $(\phi_j,v_j,\psi_j,z_j)$ of the system \eqref{eq1-pol}-\eqref{eq4-pol}.

 \begin{lemma}\label{INFO1}
	{\rm The solution $(\phi,v,\psi,z)$ of system \eqref{Polesteg1} satisfies \ma{the following estimations}}
\begin{eqnarray}
\label{CASE2-L1-eq2}
\displaystyle
\sum_{j=1}^{\infty}\|\sqrt{b}v_j\|^2=o\left(\la^{-\ell}\right),&\displaystyle
\sum_{j=1}^{\infty}\|\sqrt{d} z_j\|^2= o\left(\la^{-\ell}\right),\\[0.1in]
\displaystyle
\sum_{j=1}^{\infty}\|\la \sqrt{b} \phi_j\|^2 =o\left(\la^{-\ell}\right),&\displaystyle
\sum_{j=1}^{\infty}\|\la \sqrt{d} \psi_j\|^2=o\left(\la^{-\ell}\right).
\label{CASE2-L1-eq2phipsi}
\end{eqnarray}
\end{lemma}
\begin{proof}
First, taking the inner product of \eqref{Polesteq} with $\Phi$ in $\mathcal{H}$, we remark that 
\begin{equation}\label{L1-eq2}
\|\sqrt{b} v\|^2_{L^2(\Omega)}+\|\sqrt{d} z\|^2_{L^2(\Omega)}=-\Re\left(\left<\mathcal{A}\Phi,\Phi\right>_{\mathcal{H}}\right)=\Re\left(\left<(i\la I-\mathcal{A})\Phi,\Phi\right>_{\mathcal{H}}\right)=o(\la^{-\ell}).
\end{equation}
Thus, by the orthonormal basis decomposition, we get the first estimations in \eqref{CASE2-L1-eq2}. Now multiplying   \eqref{eq1-pol} (resp. \eqref{eq3-pol}) by $\sqrt{b}$
(resp. $\sqrt{d}$) and using the   estimations in  \eqref{CASE2-L1-eq2} and that $\|F\|_{\mathcal{H}}=o(1)$, we get the   estimations in  \eqref{CASE2-L1-eq2phipsi}. The proof has been completed.
\end{proof}

\subsection{Proof of Theorem \ref{G-P}} The proof of Theorem \ref{G-P} is divided into several Lemmas. \ma{In these following Lemmas, we assume that \eqref{CASE1} holds.}
\begin{lemma}\label{LEMMA2}
{\rm The solution $(\phi,v,\psi,z)$ of  system \eqref{Polesteg1} satisfies the following estimation
\begin{equation}\label{L2-eq1}
\sum_{j=1}^{\infty}\left(\|\phi_j'\|^2_{L^2(D_{\varepsilon})}+\mu_j^2\|\phi_j\|^2_{L^2(D_{\varepsilon})}\right)=o\left(\la^{-\left(\frac{\ell}{2}+1\right)}\right),
\end{equation}
where $D_{\varepsilon}:=(\alpha_1+\varepsilon,\alpha_4-\varepsilon)$
with a positive real number $\varepsilon$ small enough such that
$\alpha_1+\varepsilon<\alpha_2<\alpha_3<\alpha_4-\varepsilon$.}
\end{lemma}
\begin{proof}
First, let us fix the following cut-off function $\theta_1\in C^1([0,1])$, such that $0\leq \theta_1\leq 1$, for all $x\in [0,1]$ and 
\begin{equation}\label{theta1ep}
\theta_1=1\ \ \ \text{in}\ \  D_{\varepsilon}\quad \text{and}\quad \theta_1=0\ \ \ \text{in}\ \ (0,\alpha_1)\cup (\alpha_4,1).
\end{equation}
Multiplying \eqref{COMB1} by $-\theta_1 \overline{\phi_j}$, using integration by parts  over $(0,1)$, and  the definition of $b$, $c$ and $\theta_1$, we get 
\begin{equation*}
\begin{array}{l}
\displaystyle
-\int_0^1\theta_1\abs{\la \phi_j}^2dx+\int_0^1\phi'_j\left(\theta_1'\overline{\phi_j}+\theta_1\overline{\phi_j'}\right)dx+\mu_j^2\int_0^1\theta_1\abs{\phi_j}^2dx\\[0.1in]
\displaystyle
+\, \ma{i\la\int_{0}^{1}b\theta_1 \abs{\phi_j}^2dx} + i\la \int_{\alpha_2}^{\alpha_3}c \theta_1\psi_j\overline{\phi_j}dx=-\int_0^1\theta_1F_j^1\overline{\phi_j}dx.
\end{array}
\end{equation*}
Taking the real part and the sum on $j$ from 1 to $\infty$ in the above equation, we get 
\begin{equation}\label{L2-eq2}
\begin{array}{l}
\displaystyle
\sum_{j=1}^{\infty}\int_0^1\theta_1\left[\abs{\phi_j'}^2+\mu_j^2\abs{\phi_j}^2\right]dx=\sum_{j=1}^{\infty}\int_0^1\theta_1\abs{\la \phi_j}^2dx-\sum_{j=1}^{\infty}\Re\left(\int_0^1\theta_1'\phi_j'\overline{\phi_j}dx\right)\\[0.1in]
\displaystyle
-\sum_{j=1}^{\infty}\Re\left(i\la \int_{\alpha_2}^{\alpha_3}c\theta_1\psi_j\overline{\phi_j}dx\right)-\sum_{j=1}^{\infty}\Re\left(\int_0^1\theta_1F_j^1\overline{\phi_j}dx\right).
\end{array}
\end{equation}
Using Cauchy-Schwarz inequality, \eqref{L1-eq2}, the fact that $\|\Phi\|_{\mathcal{H}}=1$ and $\|F\|_{\mathcal{H}}=o(1)$, we get 
\begin{equation}\label{L2-eq3}
\left\{\begin{array}{l}
\displaystyle
\left|\sum_{j=1}^{\infty}\Re\left(\int_0^1\theta_1'\phi_j'\overline{\phi_j}dx\right)\right|\leq \max_{x\in [0,1]}|\theta_1^{'}(x)|\left(\sum_{j=1}^{\infty}\|\phi_j'\|_{L^2(\alpha_1,\alpha_4)}^2\right)^{\frac{1}{2}}\left(\sum_{j=1}^{\infty}\|\phi_j\|^2_{L^2(\alpha_1,\alpha_4)}\right)^{\frac{1}{2}}=\frac{o(1)}{\la^{\frac{\ell}{2}+1}},\\[0.1in]
\displaystyle
\left|\sum_{j=1}^{\infty}\Re\left(i\int_{\alpha_2}^{\alpha_3}c\la\theta_1\psi_j\overline{\phi_j}dx\right)\right|\leq \|c\|_{\infty}\left(\sum_{j=1}^{\infty}\|\la \psi_j\|_{L^2(\alpha_2,\alpha_3)}^2\right)^{\frac{1}{2}}\left(\sum_{j=1}^{\infty}\|\phi_j\|^2_{L^2(\alpha_2,\alpha_3)}\right)^{\frac{1}{2}}=\frac{o(1)}{\la^{\frac{\ell}{2}+1}},\\[0.1in]
\displaystyle
\left|\sum_{j=1}^{\infty}\Re\left(\int_0^1\theta_1F_j^1\overline{\phi_j}dx\right)\right|\leq \left(\sum_{j=1}^{\infty}\|F_j^1\|^2_{L^2(\alpha_1,\alpha_4)}\right)^{\frac{1}{2}}\left(\sum_{j=1}^{\infty}\|\phi_j\|^2_{L^2(\alpha_1,\alpha_4)}\right)^{\frac{1}{2}}=\frac{o(1)}{\la^{\frac{3\ell}{2}}}.
\end{array}
\right.
\end{equation}
Inserting \eqref{L2-eq3} in \eqref{L2-eq2} and  using \eqref{L1-eq2}, we get 
\begin{equation*}
\sum_{j=1}^{\infty}\int_0^1\theta_1\left[\abs{\phi_j'}^2+\mu_j^2\abs{\phi_j}^2\right]dx=o\left(\la^{-\min\left(\frac{3\ell}{2},\frac{\ell}{2}+1\right)}\right).
\end{equation*}
Finally, using the definition of the function $\theta_1$ and the fact that $\ell\geq 1$, we get the desired equation \eqref{L2-eq1}. The proof has been completed.
\end{proof}
\begin{lemma}\label{LEMMA3}
	{\rm
\ma {The solution $(\phi,v,\psi,z)$ of  system \eqref{Polesteg1} satisfies the following estimation}
\begin{equation}\label{L3-eq1}
\sum_{j=1}^{\infty}\|\la \psi_j\|^2_{L^2(\alpha_2,\alpha_3)}=\abs{a-1}o(\la^{-\frac{\ell}{2}+1})+o\left(\la^{-\min(\frac{\ell}{2},\ell-1)}\right)=\left\{\begin{array}{lll}
o\left(\la^{-\min(\frac{\ell}{2},\ell-1)}\right)&\text{if}&a=1,\\
o(\la^{-\frac{\ell}{2}+1})&\text{if}&a\neq 1.
\end{array}
\right.
\end{equation}}
\end{lemma}
\begin{proof}
First, let us fix the following cut-off function $\theta_2\in C^1([0,1])$, such that $0\leq \theta_2\leq 1$, for all $x\in [0,1]$ and 
\begin{equation}\label{theta2}
\theta_2=1\ \ \text{in}\ \ D_{2\varepsilon}:=(\alpha_1+2\varepsilon,\alpha_4-2\varepsilon)\ \ \text{and}\ \ \theta_2=0\ \ \text{in}\ \ (0,\alpha_1+\varepsilon)\cup (\alpha_4-\varepsilon,1),
\end{equation}
with a positive real number $\varepsilon$ small enough such that
$\alpha_1+2\varepsilon\leq\alpha_2<\alpha_3\leq\alpha_4-2\varepsilon$.
Multiplying \eqref{COMB1} by $a\theta_2\la \overline{\psi}_j$, using integration by parts  over $(0,1)$, and  the definition of the functions $b$, $c$ and $\theta_2$, we get 
\begin{equation}\label{L3-eq2}
\begin{array}{l}
\displaystyle
a\la^3\int_0^1\theta_2\phi_j\overline{\psi_j}dx-a\la\int_0^1\theta_2'\phi_j'\overline{\psi_j}dx-a\la\int_0^1\theta_2\phi_j'\overline{\psi_j'}dx-a\la \mu_j^2\int_0^1\theta_2\phi_j\overline{\psi_j}dx\\[0.1in]
\displaystyle
-ia \la^2  \ma{\int_{0}^{1}b\theta_2\phi_j\overline{\psi_j}dx}-ia\int_{\alpha_2}^{\alpha_3}c\abs{\la \psi}^2dx=a\la\int_0^1\theta_2F_j^1\overline{\psi_j}dx.
\end{array}
\end{equation}
Now, multiplying \eqref{COMB2} by $\theta_2\la \overline{\phi}$, using integration by parts over  $(0,1)$, and  the definition of $c$ and $\theta_2$, we get 
\begin{equation}\label{L3-eq3}
\begin{array}{l}
\displaystyle 
\la^3\int_0^1\theta_2\psi_j\overline{\phi_j}dx-a\la\int_0^1\theta_2'\psi_j'\overline{\phi_j}dx-a\la\int_0^1\theta_2\psi_j'\overline{\phi_j'}dx-a\la \mu_j^2\int_0^1\theta_2\psi_j\overline{\phi_j}dx\\[0.1in]
\displaystyle
+i\int_{\alpha_2}^{\alpha_3}c\theta_2\abs{\la \phi_j}^2dx=\la \int_0^1\theta_2F_j^2\overline{\phi_j}dx.
\end{array}
\end{equation}
Subtracting \eqref{L3-eq2} and \eqref{L3-eq3}, and taking the imaginary part, we get 
\begin{equation}\label{L3-eq4}
\begin{array}{l}
\displaystyle 
a\int_{\alpha_2}^{\alpha_3} c\abs{\la \psi_j}^2dx=(a-1)\la^3\Im\left(\int_0^1\theta_2\phi_j\overline{\psi_j}dx\right)\ma{-}a\la\Im\left(\int_0^1\theta_2'\phi_j'\overline{\psi_j}dx\right)\\[0.1in]
\displaystyle
\ma{-}a\la\Im\left(\int_0^1\theta_2'\psi_j'\overline{\phi_j}dx\right)
-a\la^2 \Im\left(i\ma{\int_{0}^{1}b\theta_2\phi_j\overline{\psi_j}dx}\right)+\int_{\alpha_2}^{\alpha_3}c\theta_2\abs{\la \phi_j}^2dx\\[0.1in]
\displaystyle 
-a\la\Im\left(\int_0^1\theta_2F_j^1\overline{\psi_j}dx\right)
-\la\Im\left( \int_0^1\theta_2F_j^2\overline{\phi_j}dx\right).
\end{array}
\end{equation}
Using integration by parts and the definition of $\theta_2$, we get 
\begin{equation}\label{L3-eq5}
-\Im\left(\int_0^1\theta_2'\phi_j'\overline{\psi_j}dx\right)=\Im\left(\int_0^1\theta_2''\phi_j\overline{\psi_j}dx\right)+\Im\left(\int_0^1\theta_2'\phi_j\overline{\psi_j'}dx\right).
\end{equation}
Inserting \eqref{L3-eq5} in \eqref{L3-eq4} and take the sum over $j$, we get 
\begin{equation}\label{L3-eq6}
\begin{array}{l}
\displaystyle 
a\sum_{j=1}^{\infty}\int_{\alpha_2}^{\alpha_3}c\abs{\la \psi_j}^2dx=(a-1)\la^3\sum_{j=1}^{\infty}\Im\left(\int_0^1\theta_2\phi_j\overline{\psi_j}dx\right)\ma{+}a\la\sum_{j=1}^{\infty}\Im\left(\int_0^1\theta_2''\phi_j\overline{\psi_j}dx\right)\\[0.1in]
\displaystyle
\ma{-}2a\la \sum_{j=1}^{\infty}\Im\left(\int_0^1\theta_2'\psi_j'\overline{\phi_j}dx\right)-a\la^2\sum_{j=1}^{\infty}\Im\left(i\ma{\int_{0}^{1}b\theta_2\phi_j\overline{\psi_j}dx}\right)+\sum_{j=1}^{\infty}\int_{\alpha_2}^{\alpha_3}c\abs{\la \phi_j}^2dx\\[0.1in]
\displaystyle
-a\la \sum_{j=1}^{\infty}\Im\left(\int_0^1\theta_2F_j^1\overline{\psi_j}dx\right)-\la\sum_{j=1}^{\infty}\Im\left( \int_0^1\theta_2F_j^2\overline{\phi_j}dx\right).
\end{array}
\end{equation}
Using Cauchy-Schwarz inequality, the definition of the function $\theta_2$, and the fact that $\|\Phi\|_{\mathcal{H}}=1$, $\|F\|_{\mathcal{H}}=o(1)$  and \eqref{L1-eq2}, we get 
\begin{equation*}
\left\{\begin{array}{ll}
\displaystyle
\left|(a-1)\la^3\sum_{j=1}^{\infty}\Im\left(\int_0^1\theta_2\phi_j\overline{\psi_j}dx\right)\right|=\abs{a-1}o(\la^{-\frac{\ell}{2}+1}),&\displaystyle
\left|a\la\sum_{j=1}^{\infty}\Im\left(\int_0^1\theta_2''\phi_j\overline{\psi_j}dx\right)\right|=o\left(\la^{-\left(\frac{\ell}{2}+1\right)}\right),\\[0.25in]
\displaystyle
\left|\la \sum_{j=1}^{\infty}\Im\left(\int_0^1\theta_2'\psi_j'\overline{\phi_j}dx\right)\right|=o\left(\la^{-\frac{\ell}{2}}\right),&\displaystyle
\left|\la^2\sum_{j=1}^{\infty}\Im\left(i\ma{\int_{0}^{1}b\theta_2\phi_j\overline{\psi_j}dx}\right)\right|=o\left(\la^{-\frac{\ell}{2}}\right),\\[0.25in]
\displaystyle
\left|a\la\sum_{j=1}^{\infty}\Im\left(\int_0^1\theta_2F_j^1\overline{\psi_j}dx\right)\right|=o\left(\la^{-\left(\ell-1\right)}\right),&\displaystyle
\left|\la \sum_{j=1}^{\infty}\Im\left( \int_0^1\theta_2F_j^2\overline{\phi_j}dx\right)\right|=o\left(\la^{-\left(\ell-1\right)}\right).
\end{array}
\right.
\end{equation*}
Finally, inserting the above estimation in \eqref{L3-eq6}, using \eqref{L1-eq2}, and the definition of the functions $c$ and $\theta_2$, we get the desired estimate \eqref{L3-eq1}. The proof has been completed. 
\end{proof} 
\begin{lemma}\label{LEMMA4}
	{\rm
		\ma{
The solution $(\phi,v,\psi,z)$ of  system \eqref{Polesteg1} satisfies the following estimation}
\begin{equation}\label{L4-1eq1}
\begin{array}{rll}
\displaystyle{\sum_{j=1}^{\infty}\left(\|\psi_j'\|^2_{L^2(\omega_{\varepsilon})}+\mu_j^2\|\psi_j\|^2_{L^2(\omega_{\varepsilon})}\right)}&=&\left\{\begin{array}{lll}
o\left(\la^{-\min(\frac{\ell}{2},\ell-1)}\right)&\text{if}&a=1,\\[0.1in]
o(\la^{-\frac{\ell}{2}+1})&\text{if}&a\neq 1.
\end{array}
\right.
\end{array}
\end{equation}
where $\omega_{\varepsilon}:=(\alpha_2+\varepsilon,\alpha_3-\varepsilon)$
with a positive real number $\varepsilon$ small enough such that $\alpha_2+\varepsilon<\alpha_3-\varepsilon$.}
\end{lemma}
\begin{proof}
let us fix the following cut-off function $\theta_3\in C^1([0,1])$, such that $0\leq \theta_3\leq 1$, for all $x\in [0,1]$ and 
$$
\theta_3=1\ \ \text{in}\ \ \omega_{\varepsilon}\quad \text{and}\quad \theta_3=0\ \ \text{in}\ \ (0,\alpha_2)\cup (\alpha_3,1). 
$$
Multiplying \eqref{COMB2} by $-\theta_3\overline{\psi_j}$, using integration by parts over $(0,1)$, we get
\begin{equation}\label{L4-eq1}
\begin{array}{l}
\displaystyle
-\int_0^1\theta_3\abs{\la \psi_j}^2dx+a\int_0^1\psi_j'\left(\theta_3'\overline{\psi_j}+\theta_3\overline{\psi_j'}\right)dx+a\mu_j^2\int_0^1 \theta_3\abs{\psi_j}^2dx\\[0.in]
\displaystyle
-i\la  \int_{\alpha_2}^{\alpha_3}c\theta_3\phi_j\overline{\psi_j}dx=-\int_0^1\theta_3F_j^2\overline{\psi_j}dx. 
\end{array}
\end{equation}
Taking the sum on $j$ in \eqref{L4-eq1}, we get 
\begin{equation}\label{L4-eq2}
\begin{array}{l}
\displaystyle
a\sum_{j=1}^{\infty}\int_0^1\theta_3\left(\abs{\psi_j'}^2+\mu_j^2\abs{\psi_j}^2\right)dx=\sum_{j=1}^{\infty}\int_0^1\theta_3\abs{\la \psi_j}^2dx-a\sum_{j=1}^{\infty}\int_0^1\theta_3'\psi_j'\overline{\psi_j}dx\\[0.1in]
\displaystyle
+i\la  \sum_{j=1}^{\infty}\int_{\alpha_2}^{\alpha_3}c\theta_3\phi_j\overline{\psi_j}dx-\sum_{j=1}^{\infty}\int_0^1\theta_3F_j^2\overline{\psi_j}dx.
\end{array}
\end{equation}
Using Cauchy-Schwarz inequality, the definition of the function $\theta_3$, the fact that $\|\Phi\|_{\mathcal{H}}=1$, $\|F\|_{\mathcal{H}}=o(1)$,  and \eqref{L1-eq2}, we get 
\begin{equation*}
\left\{\begin{array}{ll}
\displaystyle
\left|\sum_{j=1}^{\infty}\int_0^1\theta_3'\psi_j'\overline{\psi_j}dx\right|=\left\{\begin{array}{lll}
o(\la^{-\min\left(\frac{\ell}{4}+1,\frac{\ell+1}{2}\right)})&\text{if}&a=1,\\[0.1in]
o\left(\la^{-\left(\frac{\ell}{4}+\frac{1}{2}\right)}\right)&\text{if}&a\neq 1,
\end{array}
\right.
,&\displaystyle
\left|i\la \sum_{j=1}^{\infty}\int_{\alpha_2}^{\alpha_3}\theta_3\phi_j\overline{\psi_j}dx\right|=\frac{o(1)}{\la^{\frac{\ell}{2}+1}},\\[0.1in]
\displaystyle
\left|\sum_{j=1}^{\infty}\int_0^1\theta_3F_j^2\overline{\psi_j}dx\right|=o\left(\la^{-\ell}\right).
\end{array}
\right.
\end{equation*}
Finally, inserting the above estimations in \eqref{L4-eq2}, using \eqref{L3-eq1}, the definition of $\theta_3$ and the fact that $\ell\geq 4$, we get the desired result \eqref{L4-eq1}. The proof has been completed. 
\end{proof}
\begin{lemma}\label{OUT1}
	{\rm
Let $h\in C^{\infty}\left([0,1]\right)$ such that $h(0)=h(1)=0$. \ma{The solution $(\phi,v,\psi,z)$ of  system \eqref{Polesteg1} satisfies the following estimations}
\begin{eqnarray}
\sum_{j=1}^{\infty}\int_0^1h'\left(\la^2-\mu_j^2\right)\abs{\phi_j}^2dx+\sum_{j=1}^{\infty}\int_0^1h'\abs{\phi_j'}^2dx=o\left(\la^{-\frac{\ell}{2}}\right),\label{outside1}\\[0.1in]
\sum_{j=1}^{\infty}\int_0^1h'\left(\la^2-a\mu_j^2\right)\abs{\psi_j}^2dx+a\sum_{j=1}^{\infty}\int_0^1h'\abs{\psi_j'}^2dx=o\left(\la^{-\frac{\ell}{2}}\right).\label{outside2}
\end{eqnarray}}
\end{lemma}
\begin{proof}
First, multiplying \eqref{COMB1} by $-2h\overline{\phi_j'}$,  taking the real part and using the definition of $b$ and $c$, we get 
\begin{equation*}
\begin{array}{l}
\displaystyle
-2\la^2\Re\left(\int_0^1h\phi_j\phi_j'dx\right)-2\Re\left(\int_0^1h\phi_j^{''}\phi_j'dx\right)+2\mu_j^2\Re\left(\int_0^1h\phi_j\phi_j'dx\right)+2\la \Re\left(i\la\ma{\int_{0}^{1}bh\phi_j\overline{\phi_j'}dx}\right)\\[0.1in]
\displaystyle
+2\la \Re\left(i\int_{\alpha_2}^{\alpha_3}ch\psi_j\overline{\phi_j'}dx\right)=-2\Re\left(\int_0^1hF_j^1\overline{\phi_j'}dx\right).
\end{array}
\end{equation*}
Using integration by parts in the above equation and taking the sum on $j$, we get 
\begin{equation}\label{OUT1-eq1}
\begin{array}{l}
\displaystyle
\sum_{j=1}^{\infty}\int_0^1h'\left(\la^2-\mu_j^2\right)\abs{\phi_j}^2dx+\sum_{j=1}^{\infty}\int_0^1h'\abs{\phi_j'}^2dx=-2\sum_{j=1}^{\infty} \Re\left(i\ma{\la\int_{0}^{1}bh\phi_j\overline{\phi'_j}dx}\right)\\[0.1in]
\displaystyle
-2\sum_{j=1}^{\infty}\Re\left(i\la\int_{\alpha_2}^{\alpha_3}ch\psi_j\overline{\phi_j'}dx\right)-2\sum_{j=1}^{\infty}\Re\left(\int_0^1hF_j^1\overline{\phi_j'}dx\right).
\end{array}
\end{equation}
Now, using the fact that $\ma{\displaystyle{\sum_{j=1}^{\infty}\|\phi_j'\|^2=O(1)}}$ and \eqref{L1-eq2}, we get 
\begin{equation}\label{OUT1-eq2}
\left|2\sum_{j=1}^{\infty} \Re\left(i\la\int_{0}^{1}bh\phi_j\overline{\phi'_j}dx\right)\right|=o\left(\la^{-\frac{\ell}{2}}\right).
\end{equation}
Using \eqref{L2-eq1}, \eqref{L3-eq1},  the fact that $(\alpha_2,\alpha_3)\subset  D_{\varepsilon}$ and $\ell\geq 4$, we get 
\begin{equation}\label{OUT1-eq3}
\left|\sum_{j=1}^{\infty}\Re\left(i\la\int_{\alpha_2}^{\alpha_3}ch\psi_j\overline{\phi_j'}dx\right)\right|=\left\{\begin{array}{lll}
o\left(\la^{-\left(\frac{\ell}{2}+\frac{1}{2}\right)}\right)&\text{if}&a=1,\\[0.1in]
o\left(\la^{-\frac{\ell}{2}}\right)&\text{if}&a\neq 1. 
\end{array}
\right.
\end{equation}
Using the facts that $\displaystyle{\sum_{j=1}^{\infty}\|\phi_j'\|^2=O(1)}$ and that $\|F\|_{\mathcal{H}}=o(1)$, we get 
\begin{equation}\label{OUT1-eq4}
\left|\sum_{j=1}^{\infty}\Re\left(\int_0^1hF_j^1\overline{\phi_j'}dx\right)\right|=o\left(\la^{-(\ell-1)}\right).
\end{equation}
Inserting \eqref{OUT1-eq2}-\eqref{OUT1-eq4} in \eqref{OUT1-eq1} and using the fact that $\ell\geq 4$, we get \eqref{outside1}. Now, multiplying \eqref{COMB2} by $-2h\overline{\psi_j'}$, taking the real part,  integrating by parts over $(0,1)$,  taking the sum on $j$, and using the definition of $c$, we get 
\begin{equation}\label{OUT1-eq5}
\begin{array}{l}
\displaystyle
\sum_{j=1}^{\infty}\int_0^1h'\left(\la^2-a\mu_j^2\right)\abs{\psi_j}^2dx+a\sum_{j=1}^{\infty}\int_0^1h'\abs{\psi_j'}^2dx=\\[0.1in]
\displaystyle
2\sum_{j=1}^{\infty}\Re\left(i\la\int_{\alpha_2}^{\alpha_3}
ch\phi_j\overline{\psi_j'}\right)-\sum_{j=1}^{\infty}\Re\left(\int_0^1hF_j^2\overline{\psi_j'}dx\right).
\end{array}
\end{equation}
Using   the fact that $\displaystyle{\sum_{j=1}^{\infty}\|\psi_j'\|^2=O(1)}$, the definition of $F_j^2$ and the fact that $\|F\|_{\mathcal{H}}=o(1)$,  we get 
\begin{equation*}
\left|\sum_{j=1}^{\infty}\Re\left(i\lambda\int_{\alpha_2}^{\alpha_3} c\phi_j\overline{\psi_j'}dx\right) \right|=o\left(\la^{-\frac{\ell}{2}}\right)\quad \text{and}\quad \left|\sum_{j=1}^{\infty}\Re\left(\int_0^1hF_j^2\psi_j'dx\right)\right|=o\left(\la^{-\left(\ell-1\right)}\right).
\end{equation*}
Finally, inserting the above estimation and \eqref{OUT1-eq2}  in \eqref{OUT1-eq5} and using the fact that $\ell\geq 4$, we get the desired result \eqref{outside2}. The proof has been completed. 
\end{proof}
\begin{lemma}\label{OUT2}
	{\rm
Let $h\in C^{\infty}\left([0,1]\right)$ such that $h(0)=h(1)=0$. \ma{The solution $(\phi,v,\psi,z)$ of  system \eqref{Polesteg1} satisfies the following estimations}}
\begin{eqnarray}
\sum_{j=1}^{\infty}\int_0^1h'\left(-\la^2+\mu_j^2\right)\abs{\phi_j}^2dx+\sum_{j=1}^{\infty}\int_0^1h'\abs{\phi_j'}^2dx+\sum_{j=1}^{\infty}\int_0^1h^{''}\phi_j'\overline{\phi_j}dx=o\left(\la^{-(\frac{\ell}{2}+1)}\right),\label{1outside1}\\[0.1in]
\sum_{j=1}^{\infty}\int_0^1h'\left(-\la^2+a\mu_j^2\right)\abs{\psi_j}^2dx+a\sum_{j=1}^{\infty}\int_0^1h'\abs{\psi_j'}^2dx+a\sum_{j=1}^{\infty}\int_0^1h^{''}\psi_j'\overline{\psi_j}dx=o\left(\la^{-(\frac{\ell}{2}+1)}\right).\label{2outside2}
\end{eqnarray}
\end{lemma}
\begin{proof}
Multiplying \eqref{COMB1} by $-h'\overline{\phi_j}$, using integration by parts over $(0,1)$ and  taking the sum on $j$,  we get 
\begin{equation}\label{OUT2-eq1}
\begin{array}{l}
\displaystyle
\sum_{j=1}^{\infty}\int_0^1\left(-\la^2+\mu_j^2\right)h'\abs{\phi_j}^2dx+\sum_{j=1}^{\infty}\int_0^1h'\abs{\phi_j'}^2dx+\sum_{j=1}^{\infty}\int_0^1h^{''}\phi_j'\overline{\phi_j}dx\\[0.1in]
\displaystyle
+i\la \ma{ \sum_{j=1}^{\infty}\int_{0}^{1}bh'\abs{\phi_j}^2dx}+i\la \sum_{j=1}^{\infty}\int_{\alpha_2}^{\alpha_3}ch'\psi_j\overline{\phi_j}dx=-\sum_{j=1}^{\infty}\int_0^1h'F_j^1\overline{\phi_j}dx.
\end{array}
\end{equation}
Now, using \eqref{L1-eq2}, \eqref{L3-eq1}, the definition of $F_j^1$, and the fact that $\|F\|_{
\HH}=o(1)$ and $\displaystyle{\sum_{j=1}^{\infty}\|\la \phi_j\|^2=O(1)}$, we get 
\begin{equation*}
\left|\la \ma{\sum_{j=1}^{\infty}\int_{0}^{1}bh'\abs{\phi_j}^2dx}\right|=o\left(\la^{-(\ell+1)}\right),\ \ \left|\la\sum_{j=1}^{\infty}\int_{\alpha_2}^{\alpha_3}ch'\psi_j\overline{\phi_j}dx\right|=o\left(\la^{-(\frac{\ell}{2}+1)}\right),\ \ \left|\sum_{j=1}^{\infty}\int_0^1h'F_j^1\overline{\phi_j}dx\right|=o\left(\la^{-\ell}\right).
\end{equation*}
Inserting the above estimations in \eqref{OUT2-eq1} and using the fact that $\ell\geq 4$, we get \eqref{1outside1}. In the same way, multiplying \eqref{COMB2} by $-h'\overline{\psi_j}$,  using integration by parts over $(0,1)$ and using the definition of $c$, we get 
\begin{equation}\label{OUT2-eq2}
\begin{array}{l}
\displaystyle
\sum_{j=1}^{\infty}\int_0^1h'\left(-\la^2+a\mu_j^2\right)\abs{\psi_j}^2dx+a\sum_{j=1}^{\infty}\int_0^1h'\abs{\psi_j'}^2dx+a\sum_{j=1}^{\infty}\int_0^1h^{''}\psi_j'\overline{\psi_j}dx\\[0.1in]
\displaystyle
-i\la \sum_{j=1}^{\infty}\int_{\alpha_2}^{\alpha_3}ch'\phi_j\overline{\psi_j}dx=-\sum_{j=1}^{\infty}\int_0^1h'F_j^2\overline{\psi_j}dx.
\end{array}
\end{equation}
Using \eqref{L3-eq1}, the definition of $F_j^2$, and the fact that $\|F\|_{\mathcal{H}}=o(1)$ and $\displaystyle{\sum_{j=1}^{\infty}\|\la \psi_j\|=O(1)}$, we get 
\begin{equation*}
\left|i\la\sum_{j=1}^{\infty}\int_{\alpha_2}^{\alpha_3}ch'\phi_j\overline{\psi_j}dx\right|=o\left(\la^{-(\frac{\ell}{2}+1)}\right),\quad \left|\sum_{j=1}^{\infty}\int_0^1h'F_j^2\overline{\psi_j}dx\right|=o\left(\la^{-\ell}\right).
\end{equation*}
Finally, inserting the above estimations in \eqref{OUT2-eq2} and using the fact that $\ell\geq 4$, we get \eqref{2outside2}. The proof has been completed.
\end{proof}
\begin{lemma}\label{OUT-GLOB}
	{\rm
\ma{The solution $(\phi,v,\psi,z)$ of  system \eqref{Polesteg1} satisfies the following estimation}
\begin{equation}\label{OG-eq1}
\|\Phi\|^2_{\HH}=\left\{\begin{array}{lll}
o\left(\la^{-\frac{\ell}{2}+2}\right)&\text{if}&a=1,\\[0.1in]
o\left(\la^{-\frac{\ell}{2}+3}\right)&\text{if}&a\neq 1.
\end{array}
\right. 
\end{equation}}
\end{lemma}
\begin{proof}
First, adding \eqref{outside1}, \eqref{1outside1}, \eqref{outside2} and  \eqref{2outside2} we get 
\begin{equation}\label{OG-eq2}
2\sum_{j=1}^{\infty}\int_0^1h'\abs{\phi_j'}^2dx+2a\sum_{j=1}^{\infty}\int_0^1h'\abs{\psi_j'}^2dx+\sum_{j=1}^{\infty}\int_0^1h^{''}\phi_j'\overline{\phi_j}dx+a\sum_{j=1}^{\infty}\int_0^1h^{''}\psi_j'\overline{\psi_j}dx=o\left(\la^{-\frac{\ell}{2}}\right).
\end{equation}
Now, take $h(x)=x\theta_4+(x-1)\theta_5$, such that 
\begin{equation}\label{t4t5}
\theta_4:=\left\{\begin{array}{lll}
1&\text{in}&(0,\alpha_2+\varepsilon),\\[0.1in]
0&\text{in}&(\alpha_3-\varepsilon,1,)\\[0.1in]
0\leq \theta_4\leq 1&\text{in}&\omega_{\varepsilon},
\end{array}
\right.\quad \text{and}\quad \theta_5:=\left\{\begin{array}{lll}
0&\text{in}&(0,\alpha_2+\varepsilon),\\[0.1in]
1&\text{in}&(\alpha_3-\varepsilon,1),\\[0.1in]
0\leq \theta_5\leq 1&\text{in}&\omega_{\varepsilon}.
\end{array}
\right.
\end{equation}
It is easy to see that 
\begin{equation}\label{h'h''}
h'(x)=x\theta_4'+\theta_4+(x-1)\theta_5'+\theta_5\quad \text{and}\quad h^{''}=x\theta_4^{''}+2\theta_4'+(x-1)\theta_5^{''}+2\theta_5'.
\end{equation}
Using \eqref{h'h''},  the definition of $\theta_4$ and $\theta_5$,   using the fact that $\displaystyle{\sum_{j=1}^{\infty}\|\phi_j'\|^2}=O(1)$, \eqref{L2-eq1},  
 \eqref{L3-eq1},  and \eqref{L4-1eq1}, we get 
\begin{equation}\label{OG-eq3}
\left|\sum_{j=1}^{\infty}\int_0^1h^{''}\phi_j'\overline{\phi_j}dx\right|=o\left(\la^{-\left(\frac{\ell}{2}+1\right)}\right)\quad \text{and}\quad \left|\sum_{j=1}^{\infty}\int_0^1h^{''}\psi_j'\overline{\psi_j}dx\right|=\left\{\begin{array}{lll}
o\left(\la^{-\left(\frac{\ell}{2}+1\right)}\right)&\text{if}&a=1,\\[0.1in]
o\left(\la^{-\frac{\ell}{2}+\frac{1}{2}}\right)&\text{if}&a\neq 1.
\end{array}
\right.
\end{equation}
Inserting \eqref{OG-eq3} in \eqref{OG-eq2}, we get
\begin{equation}\label{OG-eq4}
\sum_{j=1}^{\infty}\int_0^1h'\abs{\phi_j'}^2dx+a\sum_{j=1}^{\infty}\int_0^1h'\abs{\psi_j'}^2dx=o\left(\la^{-\frac{\ell}{2}}\right).
\end{equation}
Setting $r=x\theta_4'+(x-1)\theta_5'$, and  using \eqref{t4t5}, \eqref{L2-eq1} and \eqref{L4-1eq1}, we get 
\begin{equation*} 
\sum_{j=1}^{\infty}\int_0^1r \abs{\phi_j'}^2dx=o\left(\la^{-(\frac{\ell}{2}+1)}\right)\quad \text{and}\quad \sum_{j=1}^{\infty}\int_0^1 r \abs{\psi_j'}^2dx=\left\{\begin{array}{lll}
o\left(\la^{-\frac{\ell}{2}}\right)&\text{if}&a=1,\\[0.1in]
o\left(\la^{-\frac{\ell}{2}+1}\right)&\text{if}&a\neq 1.
\end{array}
\right.
\end{equation*}
These estimations,  \eqref{OG-eq4} and  \eqref{h'h''} yield
\begin{equation}\label{OG-1eq5}
\sum_{j=1}^{\infty}\int_0^1(\theta_4+\theta_5)\abs{\phi_j'}^2dx=o\left(\la^{-\frac{\ell}{2}}\right)\quad \text{and}\quad \sum_{j=1}^{\infty}\int_0^1(\theta_4+\theta_5)\abs{\psi_j'}^2dx=\left\{\begin{array}{lll}
o\left(\la^{-\frac{\ell}{2}}\right)&\text{if}&a=1,\\[0.1in]
o\left(\la^{-\frac{\ell}{2}+1}\right)&\text{if}&a\neq 1.
\end{array}
\right.
\end{equation}
Using \eqref{t4t5}, \eqref{L2-eq1}, \eqref{L4-1eq1}, \eqref{OG-1eq5} and the fact that $\ell\geq 4$,  we get 
\begin{equation}\label{OG-eq5}
\sum_{j=1}^{\infty}\int_0^1\abs{\phi_j'}^2dx=o\left(\la^{-\frac{\ell}{2}}\right)\quad \text{and}\quad \sum_{j=1}^{\infty}\int_0^1\abs{\psi_j'}^2dx=\left\{\begin{array}{lll}
o\left(\la^{-\frac{\ell}{2}}\right)&\text{if}&a=1,\\[0.1in]
o\left(\la^{-\frac{\ell}{2}+1}\right)&\text{if}&a\neq 1.
\end{array}
\right.
\end{equation}
Using Poincar\'e inequality, we get 
\begin{equation}\label{OG-eq6}
\sum_{j=1}^{\infty}\int_0^1\abs{\la \phi_j}^2dx=o\left(\la^{-\frac{\ell}{2}+2}\right)\quad \text{and}\quad \sum_{j=1}^{\infty}\int_0^1\abs{\la \psi_j}^2dx=\left\{\begin{array}{lll}
o\left(\la^{-\frac{\ell}{2}+2}\right)&\text{if}&a=1,\\[0.1in]
o\left(\la^{-\frac{\ell}{2}+3}\right)&\text{if}&a\neq 1.
\end{array}
\right.
\end{equation}
Using \eqref{L2-eq1}, \eqref{L4-eq1}, \eqref{outside1}, \eqref{outside2} and \eqref{h'h''},  we get 
\begin{equation}\label{OG-eq7}
\sum_{j=1}^{\infty}\mu_j^2\int_0^1\abs{\phi_j}^2dx=o\left(\la^{-\frac{\ell}{2}+2}\right)\quad \text{and}\quad \sum_{j=1}^{\infty}\mu_j^2\int_0^1\abs{\psi_j}^2dx=\left\{\begin{array}{lll}
o\left(\la^{-\frac{\ell}{2}+2}\right)&\text{if}&a=1,\\[0.1in]
o\left(\la^{-\frac{\ell}{2}+3}\right)&\text{if}&a\neq 1.
\end{array}
\right.
\end{equation}
Finally, from \eqref{OG-eq5}-\eqref{OG-eq7}. we obtain \eqref{OUT-GLOB}. The proof has been completed. 
\end{proof}

\noindent \textbf{Proof of Theorem \ref{G-P}}. Take $\ell=4$ for $a=1$ and $\ell=6$ for $a\neq 1$ in Lemma \ref{OUT-GLOB}, we get $\|\Phi\|_{\mathcal{H}}=o(1)$, which contradicts  $\|\Phi\|_{\mathcal{H}}=1$ in \eqref{Polesteq0}. This implies that 
\begin{equation*}
\limsup_{\la\in \R,\ |\la|\to \infty}\frac{1}{\abs{\la}^{\ell}}\|(i\la I-\mathcal{A})^{-1}\|_{\mathcal{L}(\mathcal{H})}<\infty. 
\end{equation*}
with $\ell$ defined by \eqref{ell}.
Finally, according to \ma{Theorem \ref{GENN},} we obtain the desired result. The proof  has been completed. \hfill $\square$
\subsection{Proof of Theorem \ref{1G-P}} The proof of Theorem \ref{1G-P} is divided into several Lemmas. \ma{In these following Lemmas, we assume that \eqref{CASE2} holds}.

\begin{lemma}\label{Case2-LEMMA2}
	{\rm
 \ma{The solution $(\phi,v,\psi,z)$ of  system \eqref{Polesteg1} satisfies the following estimations}
\begin{equation}\label{Case2-L2-eq1}
\sum_{j=1}^{\infty}\left(\|\phi_j'\|^2_{L^2(D_{\varepsilon})}+\mu_j^2\|\phi_j\|^2_{L^2(D_{\varepsilon})}\right)=o\left(\la^{-2}\right)\quad \text{and}\quad \sum_{j=1}^{\infty}\left(\|\psi_j'\|^2_{L^2(D_{\varepsilon})}+\mu_j^2\|\psi_j\|^2_{L^2(D_{\varepsilon})}\right)=o\left(\la^{-2}\right).
\end{equation}
{\color{black}where $D_{\varepsilon}$ is defined in Lemma \ref{LEMMA2}.}}
\end{lemma}
\begin{proof}
First, multiplying \eqref{COMB1} and \eqref{COMB2} by $-\theta_1\overline{\phi_j}$ and $-\theta_1\overline{\psi_j}$  respectively {\color{black}(where $\theta_1$ is defined in \eqref{theta1ep})}, using  integration by parts over $(0,1)$,  and \eqref{CASE2-L1-eq2} and \eqref{CASE2}, and the same arguments than in the proof  of Lemma \ref{LEMMA2}, we get \eqref{Case2-L2-eq1}. The proof has been completed. 
\end{proof}
\begin{lemma}\label{CASE2-OUT1}
	{\rm
 Let $h\in C^{\infty}\left([0,1]\right)$ such that $h(0)=h(1)=0$. \ma{The solution $(\phi,v,\psi,z)$ of  system \eqref{Polesteg1} satisfies the following estimations}
\begin{eqnarray}
\sum_{j=1}^{\infty}\int_0^1h'\left(\la^2-\mu_j^2\right)\abs{\phi_j}^2dx+\sum_{j=1}^{\infty}\int_0^1h'\abs{\phi_j'}^2dx=o\left(\la^{-2}\right),\label{CASE2-outside1}\\[0.1in]
\sum_{j=1}^{\infty}\int_0^1h'\left(\la^2-a\mu_j^2\right)\abs{\psi_j}^2dx+a\sum_{j=1}^{\infty}\int_0^1h'\abs{\psi_j'}^2dx=o\left(\la^{-2}\right).\label{CASE2-outside2}
\end{eqnarray}}
\end{lemma}
\begin{proof}
Using the same technique  than the one of the proof of  Lemma \ref{OUT1}, we get the proof.
\end{proof}
\begin{lemma}\label{CASE2-OUT2}
	{\rm
Let $h\in C^{\infty}\left([0,1]\right)$ such that $h(0)=h(1)=0$. \ma{The solution $(\phi,v,\psi,z)$ of  system \eqref{Polesteg1} satisfies the following estimations}
\begin{eqnarray}
\sum_{j=1}^{\infty}\int_0^1h'\left(-\la^2+\mu_j^2\right)\abs{\phi_j}^2dx+\sum_{j=1}^{\infty}\int_0^1h'\abs{\phi_j'}^2dx+\sum_{j=1}^{\infty}\int_0^1h^{''}\phi_j'\overline{\phi_j}dx=o\left(\la^{-2}\right),\label{CASE2-1outside1}\\[0.1in]
\sum_{j=1}^{\infty}\int_0^1h'\left(-\la^2+a\mu_j^2\right)\abs{\psi_j}^2dx+a\sum_{j=1}^{\infty}\int_0^1h'\abs{\psi_j'}^2dx+a\sum_{j=1}^{\infty}\int_0^1h^{''}\psi_j'\overline{\psi_j}dx=o\left(\la^{-2}\right).\label{CASE2-2outside2}
\end{eqnarray}}
\end{lemma}
\begin{proof}
Using the same technique than the one of the proof of Lemma \ref{OUT1}, we get the proof.
\end{proof}
\begin{lemma}\label{CASE2-OUT-GLOB}
{\rm
\ma{The solution $(\phi,v,\psi,z)$ of  system \eqref{Polesteg1} satisfies the following estimation}
\begin{equation}\label{CASE2-OG-eq1}
\|\Phi\|^2_{\HH}=o\left(1\right).
\end{equation}}
\end{lemma}
\begin{proof}
Adding \eqref{CASE2-outside1}, \eqref{CASE2-1outside1}, \eqref{CASE2-outside2} and  \eqref{CASE2-2outside2} and  taking $h(x)=x\theta_4+(x-1)\theta_5$,  with $\theta_4$ and $\theta_5$  defined in \eqref{t4t5},
and using the same technique than the one of the proof of Lemma \ref{OUT-GLOB}, we get \eqref{CASE2-OG-eq1}. 
\end{proof}

\noindent\textbf{Proof of Theorem \ref{1G-P}.}  Lemma \ref{CASE2-OUT-GLOB} contradicts $\|\Phi\|_{\mathcal{H}}=1$ in \eqref{Polesteq0}. This implies that 
\begin{equation*}
\limsup_{\la\in \R,\ |\la|\to \infty}\frac{1}{\abs{\la}^{2}}\|(i\la I-\mathcal{A})^{-1}\|_{\mathcal{L}(\mathcal{H})}<\infty. 
\end{equation*}
Finally, according to \ma{Theorem \ref{GENN}}, we obtain the desired result. The proof has been completed.\hfill $\square$\\\linebreak
Some cylindrical domains $\Omega$ with particular choices for the support of  $b,  c$, and  $d$ are illustrated in Figures \ref{Fig1},  \ref{Fig2}, and  \ref{Fig3}.
\begin{figure}[h!]
	\begin{center}
		\begin{tikzpicture}
			\draw
			(0,0) coordinate (A1)-- (5,0) coordinate (B1)--(5,5) coordinate (C1)--(0,5) coordinate (D1)-- cycle;
			\draw[-,red]
			(1,0) coordinate (A2)-- (4,0) coordinate (B2)--(4,5) coordinate (C2)--(1,5) coordinate (D2)-- cycle;
			\draw[color=red,decorate,decoration={brace,amplitude=10pt}]
			(1,5.25) -- (4,5.25)  ;
			\node[blue] at (2.5,5.85){\color{red} $\text{supp} \, b\times \overline{\omega}$};
			
			\draw[-,blue]
			(2,0) coordinate (A3)-- (3,0) coordinate (B3)--(3,5) coordinate (C3)--(2,5) coordinate (D3)-- cycle;
			\draw [color=blue,decorate,decoration={brace,amplitude=10pt,mirror,raise=4pt},yshift=0pt]
			(2,0) -- (3,0) ;
			\node[] at (2.6,-0.75){\color{blue}$\text{supp} \, c\times \overline{\omega}$};
			
			\draw
			(7,0) coordinate (A4)-- (12,0) coordinate (B4)--(12,5) coordinate (C4)--(7,5) coordinate (D4)-- cycle;
			
			\draw[-,red]
			(7.75,0) coordinate (A5)-- (10.5,0) coordinate (B5)-- (10.5,5) coordinate (C5)--(7.75,5) coordinate (D5)--cycle;
			\draw[color=red,decorate,decoration={brace,amplitude=10pt}]
			(7.75,5.7) -- (10.5,5.7) ;
			\node[red,above] at (9.15,6){\color{red} $\text{supp}\, b\times \overline{\omega}$};
			
			\draw[-,cyan]
			(8.5,0) coordinate (A6)-- (11.25,0) coordinate (B6)--(11.25,5) coordinate (C6)--(8.5,5) coordinate (D6)-- cycle;
			\draw [color=cyan,decorate,decoration={brace,amplitude=10pt,mirror,raise=4pt},yshift=0pt]
			(8.5,-0.15) -- (11.25,-0.15) ;
			\node[below] at (10,-0.65){\color{cyan} $\text{supp} \, d\times \overline{\omega}$};
			
			\draw[-,blue]
			(9,0) coordinate (A6)-- (10,0) coordinate (B6)--(10,5) coordinate (C6)--(9,5) coordinate (D6)-- cycle;
			\draw[color=blue,decorate,decoration={brace,amplitude=5pt}]
			(9,5.1) -- (10,5.1) ;

			\node[blue,above] at (9.5,5.18){\scalebox{1}{$\text{supp}\, c \times \overline{\omega}$}};

		\end{tikzpicture}
	\end{center}
	\caption{Locally coupled wave equations with direct/indirect  localized viscous damping on a square.}\label{Fig1}
\end{figure}
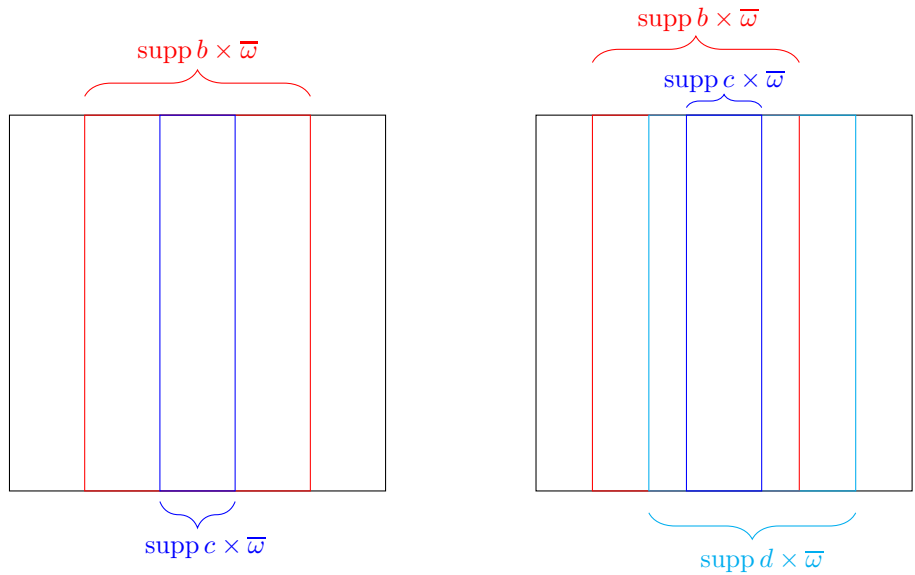

\newpage
\begin{figure}[h!]
	\begin{center}
		\begin{tikzpicture}
			
			\draw[dashed,red](1,0,0)--(5,0,0);
			\draw[dashed,red](1,0,6)--(5,0,6);
			\draw[dashed,red](1,0,6)--(1,0,0);
			\draw[dashed,red](5,0,6)--(5,0,0);
			\draw[dashed,red](5,0,6)--(5,0,0);
			\draw[dashed,red](5,0,0)--(5,6,0);
			\draw[dashed,red](5,6,0)--(1,6,0);
			\draw[dashed,red](1,6,0)--(1,0,0);
			\draw[-,red](1,6,0)--(1,6,6);
			\draw[-,red](1,6,6)--(5,6,6);
			\draw[-,red](5,6,6)--(5,6,0);
			\draw[-,red](1,6,6)--(1,0,6);
			\draw[-,red](5,6,6)--(5,0,6);
			\draw[dashed,red](2,0,0)--(4,0,0);
			
			
			
			

			\draw[dashed](0,0,0)--(6,0,0);
			\draw[dashed,black](0,0,0)--(0,6,0);
			\draw[dashed,black](0,0,0)--(0,0,6);
			\draw[-,black](0,0,6)--(6,0,6);
			\draw[-,black](6,0,6)--(6,0,0);
			\draw[-,black](6,0,6)--(6,6,6);
			\draw[-,black](6,6,6)--(6,6,0);
			\draw[-,black](6,6,0)--(6,0,0);
			\draw[-,black](6,6,0)--(0,6,0);
			\draw[-,black](6,6,6)--(0,6,6);
			\draw[-,black](0,6,6)--(0,6,0);
			\draw[-,black](0,0,6)--(0,6,6);
			\draw[dashed,blue](2,0,6)--(4,0,6);
			\draw[dashed,blue](2,0,6)--(2,0,0);
			\draw[dashed,blue](4,0,6)--(4,0,0);
			\draw[dashed,blue](4,0,6)--(4,0,0);
			\draw[dashed,blue](4,0,0)--(4,6,0);
			\draw[dashed,blue](5,6,0)--(2,6,0);
			\draw[dashed,blue](2,6,0)--(2,0,0);
			\draw[-,blue](2,6,0)--(2,6,6);
			\draw[-,blue](2,6,6)--(4,6,6);
			\draw[-,blue](4,6,6)--(4,6,0);
			\draw[-,blue](2,6,6)--(2,0,6);
			\draw[-,blue](4,6,6)--(4,0,6);

				\draw [color=blue,decorate,decoration={brace,amplitude=10pt,mirror,raise=4pt},yshift=0pt]
			(2,0,6) -- (4,0,6) ;
			\node[below] at (3.5,0,7.2){\color{blue}$\text{supp}\, c \times \overline{\omega}$};
			
				\draw [color=red,decorate,decoration={brace,amplitude=10pt,mirror,raise=4pt},yshift=0pt]
			(1.8,0,8) -- (5.8,0,8) ;
			\node[below] at (4,0,9.1){\color{red}$\text{supp}\, b \times \overline{\omega}$};

		\end{tikzpicture}
	\end{center}
	\caption{Locally coupled wave equations with indirect localized viscous damping on a cube.}\label{Fig2}
\end{figure}
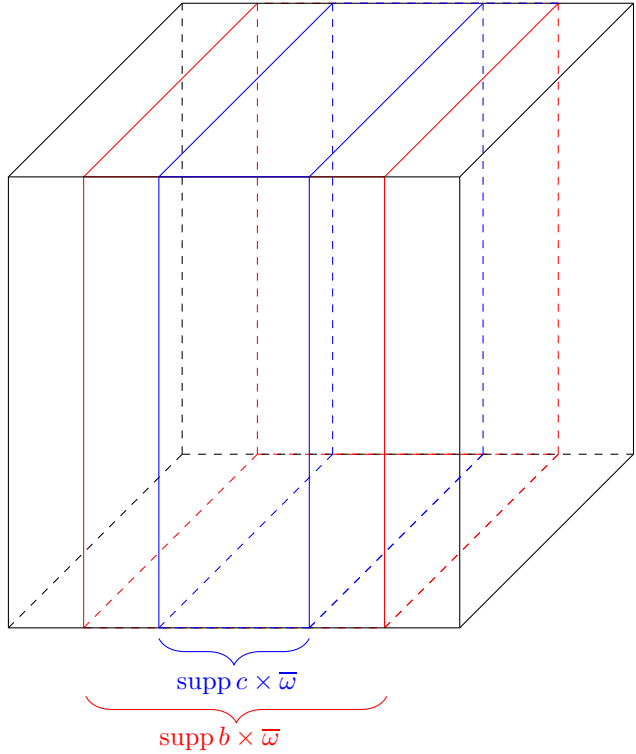
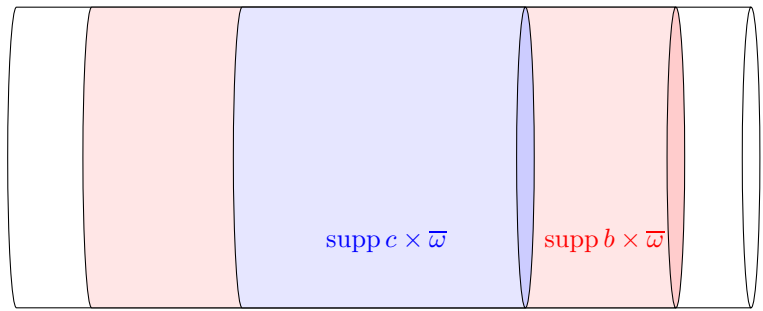
\begin{figure}[h!]
\begin{tikzpicture}
\node[cylinder, 
    draw = , 
    text = ,
    cylinder uses custom fill, 
    minimum width = 4cm,
    minimum height = 10cm] (c) at (0,0) {};
\node[cylinder, 
    draw = , 
    text = ,
    cylinder uses custom fill, 
    cylinder body fill = red!10, 
    cylinder end fill = red!20,
    minimum width = 4cm,
    minimum height = 8cm] (c) at (0,0) {};
\node[cylinder, 
    draw = , 
    text = ,
    cylinder uses custom fill, 
    cylinder body fill = blue!10, 
    cylinder end fill = blue!20,
    minimum width = 4cm,
    minimum height = 4cm] (c) at (0,0) {};

    	\node[blue,above] at (0.5,-1,1){\scalebox{1}{$\text{supp}\, c \times \overline{\omega}$}};
    		\node[red,above] at (3.4,-1,1){\scalebox{1}{$\text{supp}\, b \times \overline{\omega}$}};
 \end{tikzpicture}
\caption{Locally coupled wave equations with localized viscous damping on a cylinder.}\label{Fig3}   
\end{figure}
\section{Conclusion and open problems}
\noindent In this work, the local stabilization of  N-dimensional locally coupled wave equations on cylindrical and non regular domains is considered. The localized damping and coupling regions do not satisfy the geometric control condition ${\rm (GCC)}$. Based on the frequency domain approach with the orthonormal basis decomposition and specific multiplier techniques, we have proved a polynomial energy decay rate that depends on the speed wave propagation for indirect stabilization. For direct stability, we established a polynomial energy decay rate of order $t^{-1}$. {\color{black}The case where the coupling region is included in the damping region and both regions do not hit  the boundary  is still  an open problem (see Figure \ref{Fig5} for an illustration)}. Moreover, the case $0<\alpha_1=\alpha_2<\alpha_4<\alpha_3<1$  in \eqref{CASE2}   is also an open problem (see system (A.1) in \cite{
Wehbe2021} for the 1-dimensional case).
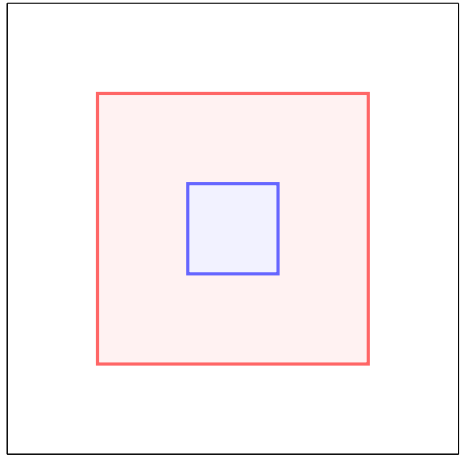
\begin{figure}[h!]
	\begin{center}
		\begin{tikzpicture}
			\draw[-](0,0)--(6,0);
			\draw[-](0,0)--(0,6);
			
			\draw[-,black](0,0)--(0,6);		
			\draw[-,black](0,0)--(6,0);
			\draw[-,black](6,0)--(6,6);
			\draw[-,black](0,6)--(6,6);				
			\draw[-,blue](2.4,2.4)--(2.4,3.6);		
			\draw[-,blue](2.4,3.6)--(3.6,3.6);	
			\draw[-,blue](3.6,2.4)--(3.6,3.6);		
			\draw[-,blue](2.4,2.4)--(3.6,2.4);		
			\draw[-,red](1.2,1.2)--(1.2,4.8);		
			\draw[-,red](1.2,4.8)--(4.8,4.8);	
			\draw[-,red](4.8,1.2)--(4.8,4.8);		
			\draw[-,red](1.2,1.2)--(4.8,1.2);
			\filldraw[color=red!60, fill=red!5, very thick] (1.2,1.2) rectangle (4.8,4.8);	
			\filldraw[color=blue!60, fill=blue!5, very thick] (2.4,2.4) rectangle (3.6,3.6);
		\end{tikzpicture}
	\end{center}
	\caption{{\color{black}The case where the coupling region (blue part) is included in the damping region (red part) and both regions are far away from the boundary.}}\label{Fig5}
\end{figure} 


\end{document}